\documentclass[a4paper, 12pt, reqno]{amsart}
\usepackage{amsmath,amsfonts,amssymb,graphics,mathrsfs,amsthm,eurosym,verbatim,enumerate,subcaption}
\usepackage[marginratio=1:1,tmargin=117pt, height=650pt]{geometry}
\usepackage[usenames, dvipsnames]{xcolor}

\usepackage[normalem]{ulem}
\usepackage{bm}
\usepackage{tikz-cd}
\usepackage[colorlinks=true,linkcolor=blue!60!black,citecolor=teal!80!black,urlcolor=RoyalBlue]{hyperref}

\numberwithin{equation}{section}
\makeatletter
\def\blfootnote{\xdef\@thefnmark{}\@footnotetext}
\makeatother
\theoremstyle{plain}
\newtheorem{theorem}{Theorem}[section]
\newtheorem{proposition}[theorem]{Proposition}
\newtheorem{cor}[theorem]{Corollary}
\newtheorem{lemma}[theorem]{Lemma}
\newtheorem{definition}[theorem]{Definition}

\newcommand*{\defeq}{\mathrel{\vcenter{\baselineskip0.5ex \lineskiplimit0pt
 \hbox{\scriptsize.}\hbox{\scriptsize.}}}%
 =}
\newtheorem*{remark}{Remark}
\newtheorem*{remarks}{Remarks}

\theoremstyle{remark}

\newcommand{\C}{{\mathbb{C}}}
\newcommand{\D}{{\mathbb{D}}}

\newcommand{\N}{{\mathbb{N}}}

\newcommand{\M}{{\mathcal{M}}}

\newcommand{\tef}{transcendental entire function}
\newcommand{\qfor}{\quad\text{for }}

\newcommand{\Rea}{\operatorname{Re }}

\begin{document}
\title[The maximum modulus set of a polynomial]{The maximum modulus set of a polynomial}
\author[{L. Pardo-Sim\'on \and D. J. Sixsmith}]{Leticia Pardo-Sim\'on \and David J. Sixsmith}
\address{Institute of Mathematics of the Polish Academy of Sciences\\ ul. \'Sniadeckich~8\\
00-656 Warsaw\\ Poland\textsc{\newline \indent \href{https://orcid.org/0000-0003-4039-5556}{\includegraphics[width=1em,height=1em]{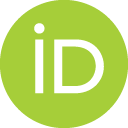} {\normalfont https://orcid.org/0000-0003-4039-5556}}}
}
\email{l.pardo-simon@impan.pl}
\address{School of Mathematics and Statistics\\ The Open University\\
Milton Keynes MK7 6AA\\ UK\textsc{\newline \indent \href{https://orcid.org/0000-0002-3543-6969}{\includegraphics[width=1em,height=1em]{orcid2.png} {\normalfont https://orcid.org/0000-0002-3543-6969}}}}
\email{david.sixsmith@open.ac.uk}
\thanks{\vspace{3pt}\vspace{3pt}\\ }
\begin{abstract}
We study the maximum modulus set, $\M(p)$, of a polynomial $p$. We are interested in constructing $p$ so that $\M(p)$ has certain exceptional features. Jassim and London gave a cubic polynomial $p$ such that $\M(p)$ has one discontinuity, and Tyler found a quintic polynomial $\tilde{p}$ such that $\M(\tilde{p})$ has one singleton component. These are the only results of this type, and we strengthen them considerably. In particular, given a finite sequence $a_1, a_2, \ldots, a_n$ of distinct positive real numbers, we construct polynomials $p$ and $\tilde{p}$ such that $\M(p)$ has discontinuities of modulus $a_1, a_2, \ldots, a_n$, and $\M(\tilde{p})$ has singleton components at the points $a_1, a_2, \ldots, a_n$.

Finally we show that these results are strong, in the sense that it is not possible for a polynomial to have infinitely many discontinuities in its maximum modulus set.
%
%
\end{abstract}
\maketitle
\section{Introduction}
Let $f$ be an entire function, and define the \emph{maximum modulus} by
\[
M(r, f) \defeq \max_{|z| = r} |f(z)|, \qfor r \geq 0.
\]
Following \cite{Sixsmithmax}, denote by $\M(f)$ the set of points where $f$ achieves its maximum modulus; we call this the \emph{maximum modulus set}. In other words
\begin{equation*}
\M(f) \defeq \{ z \in \C \colon |f(z)| = M(|z|,f) \}.
\end{equation*} 

If $f$ is a monomial, then $\M(f) = \C$; clearly this case is not interesting. Otherwise, $\M(f)$ consists of a countable union of closed \emph{maximum curves}, which are analytic except at their endpoints, and may or may not be unbounded; \cite{Blumenthal}. It is straightforward to check that the maximum modulus set is closed.

Our interest in this paper is in the case that $f$ is a polynomial. In particular, we study two ``exceptional'' features in the maximum modulus set.
%
%
The first concerns discontinuities, which we define as follows. 
\begin{definition}\normalfont 
Let $f$ be an entire function, and $r > 0$. If there exists a connected component $\Gamma$ of $\M(f)$ such that $\min \{ |z| : z \in \Gamma \} = r$, then we say that $\M(f)$ has a \emph{discontinuity of modulus} $r$. Note that a maximum modulus set may have more than one discontinuity of the same modulus.
\end{definition}


These discontinuities were first studied by Blumenthal \cite{Blumenthal}, see also \cite{retrospect}. Hardy \cite{hardy1909} was the first to  give an entire function with discontinuities in its maximum modulus set; in fact he constructed a {\tef} whose maximum modulus set has infinitely many discontinuities. Such discontinuities were studied further in \cite{letidave}, where it was shown that, given a sequence $(a_k)_{k \in \N}$ of strictly positive real numbers tending to infinity, there is a {\tef} whose maximum modulus set has discontinuities of modulus $a_k$ for each $k \in \N$. 

Blumenthal did not give any examples of a polynomial whose maximum modulus set has discontinuities, although he conjectured that there is a cubic polynomial with this property. Such a polynomial was given in \cite{jassimlondon}. Remarkably, this is the only such example in the literature, and seems to have only one discontinuity. Our first result is a significant generalisation of that in \cite{jassimlondon}, and complements the main result in \cite{letidave} mentioned above.

\begin{theorem}
\label{t1}
Suppose that $a_1, a_2, \ldots, a_n$ is a finite sequence of distinct positive real numbers. Then there exists a polynomial $p$, of degree $2n+1$, such that $\M(p)$ has discontinuities of modulus $a_1, a_2, \ldots, a_n$.
\end{theorem}
It is possible for some of the analytic curves that make up the maximum modulus set to be degenerate; in other words, to be singletons. The only examples of this behaviour are due to Tyler \cite{tyler}, who gave a {\tef} $f$ and a polynomial $p$ such that $\M(f)$ has infinitely many singleton components, and $\M(p)$ has a singleton component. We show that it is possible to significantly strengthen this polynomial case.
\begin{theorem}
\label{t2}
Suppose that $a_1, a_2, \ldots, a_n$ is a finite sequence of distinct positive real numbers. Then there exists a polynomial $p$, of degree $4n+1$, such that $\M(p)$ has singleton components at the points $a_1, a_2, \ldots, a_n$.
\end{theorem}
\begin{remarks}\normalfont
\mbox{ }
\begin{enumerate}
\item Unlike in \cite{letidave}, where the construction required complicated and delicate approximations, our results here are direct and elementary.
\item Note that we are not claiming in Theorems~\ref{t1} and \ref{t2} that there might not be additional discontinuities and/or singleton components in the maximum modulus sets; see Figure~\ref{fig1} which indicates that this indeed may happen.
\item Note also that singleton components of $\M(f)$ are always discontinuities in the sense we have defined them. Thus the conclusion of Theorem~\ref{t1} is already contained in that of Theorem~\ref{t2}. However we have retained Theorem~\ref{t1}, partly for reasons of historical interest, and partly because the degree of the polynomials is less in Theorem~\ref{t1} than in Theorem~\ref{t2}. 
\item It is natural to ask if these results can be achieved with polynomials of smaller degree. This does not seem possible with the techniques of this paper.
\end{enumerate}
\end{remarks}
Finally, we show that these constructions are strong, in the sense that a polynomial can have at most finitely many discontinuities in its maximum modulus set.
\begin{theorem}
\label{t3}
Suppose that $p$ is a polynomial. Then $\M(p)$ has at most finitely many discontinuities.
\end{theorem}
%


\subsection*{Acknowledgments}
We would like to thank Peter Strulo for programming assistance leading to Figure~\ref{fig1}.
%
%
%
\section{Proofs of Theorem~\ref{t1} and Theorem~\ref{t2}}
\label{S.main}
We require a few lemmas before proving our main results. The first is well-known, and we omit the proof. 
\begin{lemma}
\label{lem1}
If $q(z) \defeq \sum^n_{k=0}a_k z^k$ is a polynomial, then 
\begin{equation}
\label{eq_modp}
\left| q(r e^{i\theta})\right|^2= \sum_{k=0 }^n \vert a_k\vert^2r^{2k} +  \sum_{0 \leq j < k \mathop \le n} 2 \vert a_j\vert \vert a_k\vert r^{j+k} \cos((j-k)\theta+\arg(a_j)-\arg(a_k)).
\end{equation}
\end{lemma}

\begin{figure}[htb]
    \centering
    \begin{subfigure}[b]{\textwidth}
        \centering
        \includegraphics[width=0.45\linewidth]{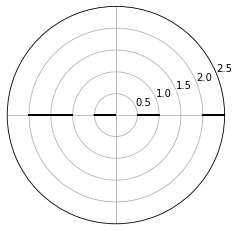}%
        \hfill
        \includegraphics[width=0.45\linewidth]{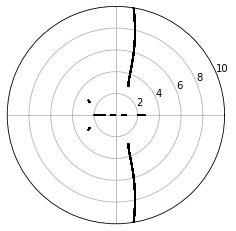}
        $p(z) \defeq 1000(z^2 + 1) +z(z^2 - 0.25)(z^2 -1)(z^2 -4).$
    \end{subfigure}
    \vskip\baselineskip
    \begin{subfigure}[b]{\textwidth}
        \centering
        \includegraphics[width=0.45\linewidth]{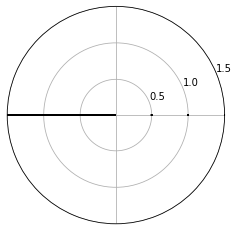}%
        \hfill
        \includegraphics[width=0.45\linewidth]{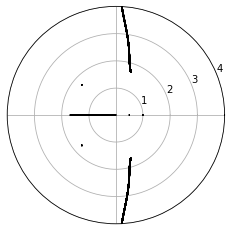}
        $\tilde{p}(z) \defeq 100(z^2 + 1)-z(z^2 - 0.25)^2(z^2 -1)^2.$
    \end{subfigure}
    \caption{\label{fig1}Computer generated graphics of $\M(p)$ and $\M(\tilde{p})$, where the former has discontinuities of modulus $0.5$, $1$ and $2$, as in Theorem~\ref{t1}, and the latter has singleton components at $0.5$ and $1$ as in Theorem~\ref{t2}. Note on the right that, when zoomed out, there appear to be additional discontinuities in these maximum modulus sets.}
\end{figure}

We use Lemma~\ref{lem1} to prove the following. Roughly speaking, this result states that we can force part of the maximum modulus set of a certain class of polynomials to lie on the real line. This result is the crux of our construction.
\begin{lemma}
\label{lem2}
Suppose that $\hat{p}$ is a polynomial with only real coefficients, and that $0 < R < R'$. For $a > 0$, set
\begin{equation}
\label{eq:pdef}
p(z) \defeq a(z^2 + 1) + \hat{p}(z).
\end{equation}
If $a$ is sufficiently large, then
\[
z \in \M(p) \text{ and } R \leq |z| \leq R' \implies \operatorname{Im} z = 0.
\] 
\end{lemma}
\begin{proof}
Let $\hat{p}$, $R$ and $R'$ be as in the statement of the lemma. Note that since $$p(-z) = a(z^2 + 1) + \hat{p}(-z),$$ and $\hat{p}$ is arbitrary, we lose no generality in proving this result only in the right half-plane. In other words, we need to prove that if $a$ is sufficiently large, then
\[
z \in \M(p) \text{ and } \Rea z \geq 0 \text{ and } R \leq |z| \leq R' \implies \operatorname{Im} z = 0.
\] 

Choose $\theta_0 \in (0, \pi/4)$. We begin by showing that if $a > 0$ is sufficiently large, then
\[
z \in \M(p) \text{ and } \Rea z \geq 0 \text{ and } R \leq |z| \leq R' \implies |\arg z| < \theta_0.
\] 

Consider first the polynomial $q(z) \defeq z^2 + 1$. Note that, by Lemma~\ref{lem1},
\[
q(r)^2 - |q(r e^{i\theta})|^2 = 2r^2(1 - \cos 2\theta).
\]
Set $\alpha = 1 - \cos 2\theta_0 > 0$. It follows that if $\theta_0 \leq |\theta| \leq \pi/2$, then
\[
q(r)^2 - |q(r e^{i\theta})|^2 \geq 2r^2\alpha.
\]
We can deduce that, if, in addition, $r \geq R$, then
\[
q(r) - |q(r e^{i\theta})| \geq \frac{r^2\alpha}{r^2 + 1} \geq \frac{R^2 \alpha}{R^2 + 1}.
\]

Let $K \defeq M(R', \hat{p})$. Choose $a > \dfrac{2K(R^2+1)}{\alpha R^2}$. We can deduce that if $r~\in~[R, R']$ and $\theta_0 \leq |\theta| \leq \pi/2$, then
\begin{align*}
p(r) - |p(r e^{i\theta})| &\geq (aq(r) - K) - (a|q(r e^{i\theta})| + K) \\
                          &= a(q(r) - |q(r e^{i\theta})|) - 2K \\
													&> 0,
\end{align*}
which establishes our first claim.

We have shown that if $a > 0$ is large enough, then the point(s) of $\M(p)$ of modulus $r \in [R, R']$ are ``close'' to the real line. It remains to show that, increasing $a$ if necessary, we can ensure that these points are in fact \emph{on} the real line.

Note, by Lemma~\ref{lem1}, that
\[
|p(r e^{i \theta})|^2 = a^2r^4 + a^2 + 2a^2r^2 \cos 2\theta + \beta(\theta),
\]
where $\beta(\theta)$ is a finite sum of terms of the form $b_k \cos k \theta$, where each $k$ is an integer, and the coefficients $b_k$ are all $O(a)$ as $a \rightarrow \infty$. Moreover, the constant in the $O(a)$ terms is independent of $r$ when we restrict ourselves to the $r$ values in the bounded set $[R,R']$. Note finally that these coefficients are positive or negative depending on whether the corresponding coefficients in $\hat{p}$ are positive or negative, though we do not use this fact.

We then have that 
\begin{equation}
\label{eq:pdiv}
\frac{\partial}{\partial \theta} |p(r e^{i \theta})|^2 = -4a^2r^2 \sin 2\theta + \beta'(\theta),
\end{equation}
and
\begin{equation}
\label{eq:pdiv2}
\frac{\partial^2}{\partial \theta^2} |p(r e^{i \theta})|^2 = -8a^2r^2 \cos 2\theta + \beta''(\theta).
\end{equation}
Note that $\beta$, $\beta'$ and $\beta''$ are all $O(a)$ as $a \rightarrow \infty$. 

Now, equation \eqref{eq:pdiv}, together with the form of $\beta$, implies that $\frac{\partial}{\partial \theta} |p(r e^{i \theta})|^2 = 0$ when $\theta = 0$, and also that
\[
\frac{1}{\theta} \frac{\partial}{\partial \theta} |p(r e^{i \theta})|^2 < -4a^2r^2 + O(a), \qfor |\theta| < \theta_0,
\]
as $a \rightarrow \infty$. It follows that, increasing $a$ if necessary, for each $r \in [R, R']$ the value $\theta = 0$ is the only stationary point of the map $\theta \mapsto |p(r e^{i \theta})|$ in the range $|\theta| \leq \theta_0$. 

Moreover, equation \eqref{eq:pdiv2}, together with the form of $\beta$, implies that
\[
\frac{\partial^2}{\partial \theta^2} |p(r e^{i \theta})|^2 < -8a^2r^2 \cos 2\theta_0 + O(a), \qfor |\theta| < \theta_0,
\]
as $a \rightarrow \infty$. Hence,  increasing $a$ one final time if necessary, we can deduce that the stationary point above is a local maximum. The result follows, using our first claim. 
\end{proof}

We use Lemma~\ref{lem2} to deduce the following.
\begin{lemma}
\label{lem3}
Suppose that $\hat{p}$ is an odd polynomial with only real coefficients, and that $0 < R < R'$. For $a > 0$, let $p$ be the polynomial defined in \eqref{eq:pdef}. If $a$ is sufficiently large, then the following holds. Suppose that $R \leq r \leq R'$. Then:
\begin{equation*}
\M(p)\cap \{ z \in \C : |z| = r \} = \renewcommand{\arraystretch}{1.5}\left\{\begin{array}{@{}l@{\quad}l@{}}
\{ r \} & \text{ if } \hat{p}(r) > 0, \\ 
\{ -r \} & \text{ if } \hat{p}(r) < 0, \\
\{-r,r\}  & \text{ if } \hat{p}(r) = 0.
\end{array}\right.\kern-\nulldelimiterspace
\end{equation*}
\end{lemma} 
\begin{proof}
We first choose $a > 0$ large enough that the consequence of Lemma~\ref{lem2} holds. Since $\hat{p}$ is odd, we have that
\[
p(\pm r) = a(r^2 + 1) \pm \hat{p}(r).
\]
The result then follows easily.
\end{proof}

We can now prove our two main constructions.
\begin{proof}[Proof of Theorem~\ref{t1}]
Let $a_1, a_2, \ldots, a_n$ be distinct positive real numbers as in the statement of the Theorem. Let $\hat{p}$ be the odd polynomial
\[
\hat{p}(z) \defeq z(z^2 - a_1^2)(z^2 - a_2^2) \ldots (z^2 - a_n^2),
\]
and set 
\begin{equation}
\label{eq:Rdef}
R \defeq \frac{1}{2}\min \{ a_1, a_2, \ldots, a_n \} \text{ and } R' \defeq 2 \max \{ a_1, a_2, \ldots, a_n \}.
\end{equation}
Note that $\hat{p}(r)$ changes sign, as $r$ increases from zero, every time we pass through one of the $a_k$. Let $a > 0$, and let $p$ be the polynomial in \eqref{eq:pdef}. The result then follows from the comment above, together with Lemma~\ref{lem3}.
\end{proof}
\begin{proof}[Proof of Theorem~\ref{t2}]
Let $a_1, a_2, \ldots, a_n$ be distinct positive real numbers as in the statement of the Theorem. Let $\hat{p}$ be the odd polynomial
\[
\hat{p}(z) \defeq -z(z^2 - a_1^2)^2(z^2 - a_2^2)^2 \ldots (z^2 - a_n^2)^2,
\]
and set $R$ and $R'$ as in \eqref{eq:Rdef}. Note that if $r > 0$, then $\hat{p}(r)$ is strictly negative, except when $r = a_k$, for some $k$, in which case $\hat{p}(r) = 0$. Let $a > 0$, and let $p$ be the polynomial in \eqref{eq:pdef}. The result then follows from the comment above, together with Lemma~\ref{lem3}.
\end{proof}
%
%
%
\section{Proof of Theorem \ref{t3}}
In this section, we make use of the pioneering work of Blumenthal on $\M(f)$ for any entire map $f$. The results that we require are summarized in the following theorem.  
\begin{theorem}[\cite{Blumenthal}]\label{thm_maxcurves} Let $f$ be an entire function, and let $S\subset \C$ be a compact set. Then $\M(f) \cap S$ is either empty, or consists of a finite number of closed curves, analytic except at their endpoints, and which can intersect in at most finitely many points. 
\end{theorem}

\begin{remark}\normalfont
Proofs of the results in \cite{Blumenthal} can also be found in \cite[II.3]{valironlectures}. We note that they are based on the study of the set of points where local maxima of the map $\theta \mapsto |f(re^{i \theta})|$ occur. The local structure of these points consists of a (finite) collection of analytic arcs. See also the work of Hayman \cite{Hayman_origin}.
\end{remark} 

\begin{cor} \label{cor_bounded_dis} Let $f$ be an entire function, and let $S\subset \C$ be any compact set. Then $\M(f) \cap S$ has at most finitely many discontinuities. 
\end{cor}

\begin{proof} Since $\M(f)$ is a collection of closed curves, by definition of discontinuity, there is a bijection from the set of discontinuities of $\M(f)$ %
to the set of all connected components of $\M(f)$ that do not contain the point zero. By this, and since by Theorem \ref{thm_maxcurves} 
$\M(f) \cap S$ has finitely many components, the result follows. 
\end{proof}

\begin{remark}\normalfont
Note that it follows easily from Corollary~\ref{cor_bounded_dis} that the maximum modulus set of a {\tef} can have at most countably many discontinuities.
\end{remark}
 
\begin{proposition}\label{prop_reciprocal}
Suppose that $p$ is a polynomial of degree $n$, and define its \emph{reciprocal polynomial}, $q$, by $q(z) \defeq z^n p(1/z)$. Then $w \in \M(q)\setminus\{0\}$ if and only if $1/w \in \M(p)\setminus\{0\}$.
\end{proposition}
\begin{proof}
Since the reciprocal of the reciprocal of a polynomial equals the original polynomial, it suffices to prove one direction. Suppose that $w \in \M(q)\setminus\{0\}$. Then 
\[
|q(w)| = \max_{|z| = |w|} |q(z)|,
\]
and so
\[
|p(1/w)| = \max_{|z| = |w|} |p(1/z)|.
\]
Hence $1/w \in \M(p)\setminus\{0\}$, as required.
\end{proof}

\begin{proof}[Proof of Theorem \ref{t3}] Let $p$ be a polynomial of degree $n$ and let $q$ be its reciprocal polynomial. Denote by $\overline{\D}$ the closure of the unit disk centred at the origin. Then, by Corollary \ref{cor_bounded_dis}, both $\M(p) \cap \overline{\D}$ and $\M(q) \cap \overline{\D}$ have at most finitely many discontinuities. Thus, by Proposition \ref{prop_reciprocal}, $\M(p) \setminus \D$ also has at most finitely many discontinuities, and the result follows. 
\end{proof}


%
%
\bibliographystyle{alpha}
\bibliography{MaxModPolyReferences}
\end{document}